\documentclass{amsart}
\usepackage{amsmath,enumerate, amssymb}
\usepackage[dvips]{graphicx,color}
\allowdisplaybreaks
\newtheorem{Theorem}{Theorem}[section]
\newtheorem{Lemma}[Theorem]{Lemma}

\newtheorem{Corollary}[Theorem]{Corollary}
\newtheorem{Proposition}[Theorem]{Proposition}
\newtheorem{Exm}{Example}[section]
\newtheorem{Remark}{Remark}[section]

\title{Some results on Strongly Operator Convex Functions and Operator Monotone Functions}
\date{December 18, 2017}
\author{Lawrence G. Brown,  Mitsuru Uchiyama}
\address{Department of Mathematics, Purdue Univ. West Lafayette, Indiana}
\email{lgb@math.purdue.edu}
\address{Department of Mathematics, Shimane Univ. Matsue, and Ritsumeikan Univ. Otsu, Japan} 
\email{uchiyama@riko.shimane-u.ac.jp}
\thanks{\\
The second named author was supported in part by (JSPS) KAKENHI 17K05286}
\begin{document} 


\begin{abstract}
This paper concerns three classes of real-valued functions on intervals, operator monotone functions, operator convex functions, and strongly operator convex functions.  Strongly operator convex functions were previously treated in [3] and [4], where 
operator algebraic semicontinuity theory or operator theory were substantially used.  In this paper we provide an alternate treatment that uses only operator inequalities (or even just matrix inequalities).  We show also that if $t_0$ is a point in the domain of a continuous function $f$, then $f$ is operator monotone if and only if $(f(t)-f(t_0)/(t-t_0)$ is strongly operator convex.  Using this and previously known results, we provide some methods for constructing new functions in one of the three classes from old ones.  We also include some discussion of completely monotone functions in this context and some results on the operator convexity or strong operator convexity of $\varphi\circ f$ when $f$ is operator convex or strongly operator convex.\\
\end{abstract}
\maketitle

\section{Introduction}
 Let $f(t)$ be a real continuous function defined on a non-degenerate interval $J$ 
in the real axis. 
For a bounded self-adjoint operator (or matrix) $A$
 on a Hilbert space ${\bf H}$ whose spectrum is in $J$, $f(A)$ is well-defined. 
Then $f$ is called an {\it operator monotone function} on 
$J$, denoted by $f\in {\bf P}(J)$, if $f(A)\leqq f(B)$, whenever $A\leqq B$. 
We call $f$ {\it operator decreasing} if $-f$ is operator
 monotone. 
The L\"{o}wner theorem \cite{L-1}
says that a $C^1$-function $f$ 
is operator monotone on an open interval $J$ if and only if 
the L\"{o}wner kernel function $K_f(t,s)$ defined by 
$$K_f(t,s)=\frac{f(t)-f(s)}{t-s}\quad(t\ne s), \quad K_f(t,t)=f'(t),$$ 
is positive semi-definite on $J$, and that such a function $f$ 
possesses a holomorphic extension $f(z)$ into the open upper 
half plane $\Pi_+$ 
which maps $\Pi_+$ into itself (unless $f$ is constant), namely $f(z)$ is a {\it Pick function}. 
Since $f(z)$ also has a holomorphic extension to the open lower 
half plane $\Pi_-$, then $f(t)$ has a holomorphic extension to 
 $J\cup \Pi_+\cup \Pi_-$. 
In this case, it follows from Herglotz's theorem that $f(t)$ has an integral representation:
\begin{equation}\label{eq:1-0}
f(t)= \alpha + \beta t + \int_{-\infty}^{\infty}
(\frac{1}{x-t}-\frac{1}{x-t_0})d\nu(x),
\end{equation}
where $t_0$ is any point in $J$, $\alpha$ is real and depends on the choice of $t_0$, $\beta\geqq 0$ and  $\nu $ is a Borel measure so that 
$$\int_{-\infty}^{\infty} \frac{1}{x^2+1} d\nu(x)<\infty, \quad \nu(J)=0.$$
 For further details see \cite{B-0, D, H-J, R}. \par
 A continuous function $f$ defined on $J$ is called an 
 {\it operator convex function} on $J$ if 
$f(sA+(1-s)B)\leqq s f(A)+(1-s)f(B)$ for every $0<s<1$ and for every pair $A, B$ with spectra in $J$. 
An {\it operator concave function} is similarly defined. 
Bendat-Sherman \cite{Bend-S} have shown that a $C^1$-function $g(t)$
on an open interval $J$ is operator convex if and only if $K_g (t, t_0)$ 
is operator monotone on $J$ for every $t_0\in J$.
In \cite{U-2010} it was proved that $g(t)$ is 
operator convex if $K_g (t, t_0)$ 
is operator monotone for one point $t_0\in J$.
We remark that if $g(t)$ is operator convex on an open interval $J$, then 
$g(t)$ has a holomorphic extension to  $J\cup \Pi_+\cup \Pi_-$. \par
Davis \cite{Davis} has shown that $g$ is an operator convex function on $J$ 
 if and only if $Pg(PAP)P \leq Pg(A)P$ for every $A$ with spectrum in $J$ and 
for every orthogonal projection $P$. 
In \cite{B-2} Brown called $g$ a {\it strongly operator convex function} 
if $P g (PAP) P \leq g(A)$ for every $A$ with spectrum in $J$ and 
for every orthogonal projection $P$; actually, he first investigated such functions, without naming them, 
in \cite[Theorem 2.36]{B-1}. We write here $g\in {\bf SOC}(J)$ if $g$ is strongly operator convex.
It is clear that a strongly operator convex function is operator convex 
and that the identity function $f(t)=t$ is not 
strongly operator convex on any interval.  \par

In both \cite{B-1} and \cite{B-2} it was shown that there are many equivalent characterizations of strongly operator convex functions.  These include operator inequalities, a global condition, an integral representation, and a differential condition, as well as more technical conditions involving operator algebras or operator theory.  We will show the most important of these equivalences in the framework of operator inequalities.  We also show
 that  if  $f(t)$ is a $C^1$-function on $J$ and $t_0\in J$, then 
$f(t)$ is operator monotone if and only if $K_f(t, t_0)$ is strongly operator convex. 
By making use of this we will give some methods to construct new operator monotone 
functions from old ones and prove that any $g$ in ${\bf SOC}(0, \infty)$ is a completely monotone function. 

\section{Preliminary Results}
Let $g(t)$ be a continuous function on $J$ and $\lambda$ a real number. Then 
it is clear that  $g- \lambda \in{\bf SOC}(J)$  if and only if 
\begin{equation*}
Pg(PAP)P + \lambda (I-P) \leq g(A)   
\end{equation*}
for every bounded self-adjoint operator $A$ with spectrum in $J$ and 
for every orthogonal projection $P$.

The following proposition is essentially treated in 
\cite{B-2}; in particular the condition (iii) is explicitly stated. However it seems to be worthwhile to give a different proof. 
\begin{Proposition}\label{Th:1}
Let $g(t)$ be a non-constant continuous function on $J$. Then the following are equivalent: 
\begin{itemize}
\item[(i)] $g - \lambda\in {\bf SOC}(J)$.
\item[(ii)]
$g(t)>\lambda$ and 
\begin{align*}&\frac{1}{2} g(A) + \frac{1}{2} g(B) - g(\frac{A + B}{2})\\
 &\geq  
\frac{1}{4} \left( g(A) - g(B)\right) \{\frac{1}{2} g(A) + \frac{1}{2} g(B) -  \lambda I\}^{-1} 
\left( g(A) - g(B)\right)\end{align*}
for every pair of bounded self-adjoint operators $A, B$ with spectra in $J$.
\item[(iii)]
$g(t)>\lambda$ and  
\begin{align*}& 
s g(A) + (1-s) g(B) - g(s A + (1-s) B) \\
&\geq 
s(1-s) \left( g(A) - g(B)\right) \{s g (B) 
+ (1-s) g(A) -  \lambda I \}^{-1} 
\left( g(A) - g(B)\right)\end{align*}
for $0<s<1$ and for every pair of bounded self-adjoint operators $A, B$ with spectra in $J$.

\item[(iv)]
$g(t)>\lambda$ and 
\begin{align*} &S^* g(A) S + \sqrt{I - S^*S} g(B) \sqrt{I-S^*S}- 
g(S^* A S + \sqrt{I - S^*S} B \sqrt{I - S^*S}) \\
 &\geq  X
 \{ \sqrt{I-SS^*} g (A) \sqrt{I-SS^*} + 
S g (B) S^* -  \lambda I\}^{-1} X^*
\end{align*}
for every contraction $S$ and for every pair of bounded self-adjoint operators 
$A, B$ with spectra in $J$, 
where $X= S^* g(A) \sqrt{I - SS^*} - \sqrt{I-S^*S} g(B) S^*$.
\end{itemize}
\end{Proposition}
\begin{proof}
(i) $\Rightarrow$ (ii). Define a unitary operator $W$ and a projection $P$ on 
${\bf H} \oplus {\bf H}$ by 
$$W=\frac{1}{\sqrt{2}} \begin{pmatrix} I & -I \\ I & I \end{pmatrix}, \; 
P= \begin{pmatrix} I & 0 \\ 0 & 0 \end{pmatrix}.
$$ Then we have 
$ Pg \left(P W^*(A\oplus B) W P\right)P  + \lambda I \leq 
g \left(W^* (A\oplus B) W\right)$, from which it follows that 
\begin{equation}\label{eq:1-2}
 \begin{pmatrix}g(\frac{A+B}{2}) & 0\\ 0 & \lambda I \end{pmatrix} \leq  
\frac{1}{2} \begin{pmatrix} g(A) + g(B) & g(B)- g(A) \\ g(B) - g(A) & g(A) + g(B)
\end{pmatrix}.
\end{equation}
 This shows that $g(t)\geq \lambda $ and $g(t)$ is 
operator convex. Assume $g (a) =\lambda$ for $a\in J$. Substituting $a I$ and $t I$ for 
$A$ and $B$, respectively, in the above inequality yields 
$$(g(t) - \lambda)(g(\frac{a+t}{2}) - \lambda) \leq 0.$$
Since $g(t)\geq \lambda$ and $g(t)$ is convex, $g(\frac{a+t}{2})=\lambda$ 
for every $t$, no matter whether $g(t)=\lambda$ or $g(t)>\lambda$. Repeat 
this procedure to get $g =\lambda$. This is inconsistent with the assumption. 
We therefore get $g(t)>\lambda$ for $t\in J$. The operator 
inequality in (ii) follows from $\eqref{eq:1-2}$. \\
(ii) $\Rightarrow$ (iv). For a contraction $S$ on ${\bf H}$
 define unitary operators $U, \;V$ on ${\bf H} \oplus {\bf H}$ by 
$$U=\begin{pmatrix} S & \sqrt{I-SS^*}\\ \sqrt{I-S^*S} & - S^*\end{pmatrix}, \;
V= \begin{pmatrix}S & -\sqrt{I-SS^*}\\ \sqrt{I-S^*S} &  S^*\end{pmatrix}.$$
By (ii) we have 
\begin{align*} &\frac{1}{2} \{U^* (g(A)\oplus g(B))U + (V^*(g(A)\oplus g(B))V\} -
g(\frac{1}{2} U^*(A\oplus B)U +\frac{1}{2}V^*(A\oplus B)V)\\
& \geq \frac{1}{4} Y \{\frac{1}{2}U^* (g(A)\oplus g(B))U + \frac{1}{2}(V^*(g(A)\oplus g(B))V - 
\lambda( I\oplus I)\}^{-1}Y,
\end{align*}
where $Y=U^* (g(A)\oplus g(B))U - (V^*(g(A)\oplus g(B))V$.
Since 
\begin{align*}&\frac{1}{2} U^*(A\oplus B)U +\frac{1}{2}V^*(A\oplus B)V 
=\bigl(\begin{smallmatrix}S^* A S + \sqrt{I - S^*S} B \sqrt{I - S^*S}& 0\\
0&  \sqrt{I-SS^*}A \sqrt{I-SS^*} + S B S^*\end{smallmatrix}\bigr), \\
&\frac{1}{2} U^*(A\oplus B)U -\frac{1}{2}V^*(A\oplus B)V
= \bigl(\begin{smallmatrix} 0 & S^* A \sqrt{I - SS^*} - \sqrt{I-S^*S} B S^*\\
\sqrt{I - SS^*}AS - SB\sqrt{I-S^*S} & 0 \end{smallmatrix}\bigr), 
\end{align*}
we obtain 
\begin{align*} &\bigl(\begin{smallmatrix}S^* g(A) S + \sqrt{I - S^*S} g(B) \sqrt{I-S^*S}
& 0\\
0 & \sqrt{I-SS^*} g (A) \sqrt{I-SS^*} + S g (B) S^*\end{smallmatrix}\bigr) \\
-& \bigl(\begin{smallmatrix} g(S^* A S + \sqrt{I - S^*S} B \sqrt{I - S^*S}) & 0  \\
0 & g(\sqrt{I-SS^*}A \sqrt{I-SS^*} + S B S^*) \end{smallmatrix}\bigr)\\
\geq & \bigl(\begin{smallmatrix}0&X\\ X^*& 0\end{smallmatrix}\bigr)
\bigl(\begin{smallmatrix}  S^* g(A) S + \sqrt{I - S^*S} g(B) \sqrt{I-S^*S} -\lambda I & 0\\
0& \sqrt{I-SS^*} g (A) \sqrt{I-SS^*} + S g (B) S^* -\lambda I\end{smallmatrix}\bigr)^{-1} 
\bigl(\begin{smallmatrix}0&X\\ X^*& 0\end{smallmatrix}\bigr).
\end{align*}
By comparing the $(1, 1)$ elements on both sides we derive the required inequality. \\
(iv)$\Rightarrow$(i). Put $S=P$ and $B=t I\; (\forall t \in J)$ in (iv). Then we get 
\begin{align*}
&Pg(A)P + g(t)(I-P) - g(PAP + t(I-P)) \\
\geq &Pg(A)(I-P)\{(I-P)g(A)(I-P)+g(t)P -\lambda I\}^{-1}(I-P)g(A)P,
\end{align*}
and hence 
\begin{align*}
&Pg(A)P -  Pg(PA P)P \\
\geq &Pg(A)(I-P)\{(I-P)g(A)(I-P) -\lambda (I-P)\}^{-1}(I-P)g(A)P.
\end{align*}
This implies
\begin{align*}
&g(A) - Pg(PAP)P+ \lambda (I-P)\\
=&\begin{pmatrix}Pg(A)P -  Pg(PAP)P&Pg(A)(I-P)\\
(I-P)g(A)P&(I-P)g(A)(I-P) -\lambda (I-P)\end{pmatrix}\geqq 0.
\end{align*} 
We therefore obtain (i). 
Clearly (iv)$\Rightarrow$(iii)$\Rightarrow$(ii).
\end{proof}
\begin{Exm}\rm
We write $A>0$ if $A\geq 0$ and $A$ is invertible. Recall the following equality from \cite [page 555]{H-J}:
for $A>0, \;B>0$ and for $0<s<1$
\begin{align}\label{eq:2} 
&(s A^{-1} + (1-s) B^{-1}) - (s A + (1-s) B)^{-1}\notag \\
 =& s(1-s)(A^{-1} - B^{-1})((1-s)A^{-1}
+ s B^{-1})^{-1}  (A^{-1} - B^{-1}).
\end{align}
From this equality and Proposition \ref{Th:1},  $g(t)=1/t$ turns out to be 
strongly operator convex on $(0, \infty)$.
\end{Exm}
\begin{Remark}\rm
Even if $g\in {\bf SOC}(J)$ and $g(t)>\lambda$, $g - \lambda $ is 
not necessarily in ${\bf SOC}(J)$.
We give a counterexample. Note that $1/t\in {\bf SOC}(0, 1)$ and $1/t > 1$ on $(0,1)$. 
But $1/t -1 \not\in {\bf SOC}(0, 1)$. 
Because, assume $ 1/t -1 \in {\bf SOC}(0, 1)$. Then by Proposition\ref{Th:1}(ii),
  for $0<A, B <1$
\begin{align*}
&\frac{1}{2} (A^{-1} +  B^{-1}) - 2(A + B)^{-1} \\
\geqq &\frac{1}{4}(A^{-1} - B^{-1})\bigl(\frac{1}{2}(A^{-1}
+  B^{-1}) - I\bigr)^{-1}  (A^{-1} - B^{-1}).
\end{align*}
But, by $\eqref{eq:2}$ the left side equals 
$\frac{1}{4}(A^{-1} - B^{-1})\bigl(\frac{1}{2}(A^{-1}
+  B^{-1})\bigr)^{-1}  (A^{-1} - B^{-1})$. Therefore the above inequality does not hold 
if $A^{-1} - B^{-1}$ is invertible.
\end{Remark}
\begin{Lemma}\label{Le:1}\rm(\cite[Theorem 2.36]{B-1} or \cite[Theorem 1.2]{B-2}). \it $g\in {\bf SOC}(J)$ if and only if 
$g=0$, or $g(t)>0, \forall t\in J$ and $1/g(t)$ is operator concave. 
\end{Lemma}
\begin{proof}
We may assume $g(t)>0$. Put $s=1/2$ and substitute $(1/g)(A)$ and $(1/g) (B)$ for 
$A$ and $B$ in $\eqref{eq:2}$, respectively. Then we get 
\begin{align*}
&\frac{1}{2}(g(A)+g(B))-\bigl(\frac{g(A)^{-1} + g(B)^{-1}}{2}\bigr)^{-1}\\
=&\frac{1}{4}
\bigl(g(A)-g(B)\bigr)\bigl(\frac{g(A) + g(B)}{2} \bigr)^{-1}\bigl(g(A)-g(B)\bigr).
\end{align*}
On the other hand, by Proposition\;\ref{Th:1}(ii), 
$g\in {\bf SOC}(J)$ if and only if 
\begin{align*}&\frac{1}{2} g(A) + \frac{1}{2} g(B) - g(\frac{A + B}{2})\\
 \geq&  
\frac{1}{4} \left( g(A) - g(B)\right)\bigl(\frac{g(A) + g(B)}{2} \bigr)^{-1}
\left( g(A) - g(B)\right).\end{align*}
In view of the above equality, this is equivalent to 
$$ g(\frac{A+B}{2})^{-1}\geqq \frac{1}{2}\bigl( g(A)^{-1} + g(B)^{-1}\bigr);$$
i.e., $1/g(t)$ is operator concave on $J$. 
\end{proof}

It is well-known that $g(t)$ is operator decreasing (or operator monotone) 
on an infinite interval 
$(a, \infty)$ if and only if $g(t)$ is operator convex (or operator concave) and
 $g(\infty)<\infty$ (or $g(\infty) >-\infty$). For a function on $(-\infty, b)$ 
we can get the symmetric result. It is also known that for 
$f\in {\bf P}(a, b)$ there is a decomposition of $f(t)$ such that  
$f(t)=f_+(t) + f_-(t) \quad (a<t<b)$, 
where $f_+\in {\bf P}(a, \infty)$ and $f_-\in {\bf P}(-\infty, b)$ (cf. \cite{U-2010}).

Before proceeding to the next proposition, recall that 
$t^2$ is operator convex on $(-\infty, \infty)$. 
Proposition\;\ref{Th:2}(iii) has been essentially proved in (the proof of) Proposition 2.39 
of \cite{B-1}.
\begin{Proposition}\label{Th:2}
\begin{itemize}
\item[(i)] $g\in {\bf SOC}
(-\infty, \infty)$ if and only if $g(t)$ is a non-negative 
constant function. 
\item[(ii)] Let $g(t)$ be a non-constant continuous function on $(a, \infty)$ with $-\infty<a$. 
Then  $g \in {\bf SOC}(a, \infty)$
if and only if  $g(t)>0$ and $g(t)$ is operator decreasing. In this case, $g(t)$ is represented as 
$$g(t)= g(\infty) + \int_{(-\infty, a]}\frac{1}{t-x}d\nu_- (x),\quad \int_{(-\infty, a]}\frac{1}{|x| + 1}d\nu_- (x)<\infty.$$ 
\item[(iii)] Let $g(t)$ be a non-constant continuous function on $(-\infty, b)$ with $b<\infty$. Then 
$g\in {\bf SOC}(-\infty, b)$ if and only if $g(t)>0$ and $g(t)$ is operator monotone.
In this case, $g(t)$ is represented as 
$$g(t)= g(-\infty) + \int_{[b, \infty)}\frac{1}{x-t}d\nu_+ (x),\quad 
\int_{[b, \infty)}\frac{1}{|x| + 1}d\nu_+ (x)<\infty.$$ 
\end{itemize}
 \end{Proposition}
\begin{proof} (i). By Lemma\;\ref{Le:1},  $g\in {\bf SOC}(-\infty, \infty)$ and $g\ne 0$ if and only if  
$g(t)>0$ and $1/g(t)$ is a positive operator concave function on $(-\infty, \infty)$; this implies that 
$g(t)$ is constant.\\
(ii). Note that $g\in {\bf SOC}(a, \infty)$ and $g\ne 0$ if and only if  
$g(t)>0$ and $1/g(t)$ is an operator concave function on$(a, \infty)$, which implies 
the operator monotonicity of $1/g(t)$. This implies $g(t)$ is operator decreasing. 
Represent $-g(t)$ by $\eqref{eq:1-0}$ and write $\nu_-$ instead of $\nu$. 
Since $-g(t)<0$,  we get  $\beta=0$ and 
$-g(\infty)= \alpha + \int_{(-\infty, a]}- \frac{1}{x - t_0} d\nu_-(x)$. This yields 
$$-g(t)= -g(\infty) + \int_{(-\infty, a]} \frac{1}{x-t}d\nu_-(x), \quad 
\int_{(-\infty, a]} \frac{1}{|x|+1} d\nu_-(x)<\infty.$$
We therefore obtain the required formula. 
We can see (iii) analogously.
\end{proof}  
 
The following corollary is easy to see and useful to construct strongly operator 
convex functions.  
\begin{Corollary}\label{Cor:1} Let $f(t)$ be an increasing continuous 
function on $(0, \infty)$ 
with $\lambda\!:=f(0+)>-\infty$. Then the following are equivalent: \\
(i)  $f\in {\bf P}(0, \infty)$,\; (ii) $f(\frac{1}{t})- \lambda\in {\bf SOC}(0, \infty)$, 
(iii) $f(-\frac{1}{t})- \lambda \in {\bf SOC}(-\infty, 0)$.
\end{Corollary}

Recall that a function $h(t)\in C^{\infty}(0, \infty)$ is called a \it completely monotone 
function \rm if $(-1)^n h^{(n)}(t) \geqq 0\; (0<t<\infty)$ for $n=0, 1, 2, \cdots$.   

\begin{Proposition}\label{Th:8}
\begin{itemize}
\item[(i)] Any $g$ in ${\bf SOC}(0, \infty) $ is a completely monotone function.
\item[(ii)] If $f(t)>0$ is an operator monotone function on $(0, \infty)$, then 
$1/f(t)$ is a completely monotone function.  
\item[(iii)] If $f(t)$ is an operator monotone function on $(0, \infty)$, then 
$f(t)$ is a Bernstein function; i.e., $(-1)^{n-1} f^{(n)}(t)\geq 0\; (0<t<\infty)$ 
for $n=1,2, \cdots$. 
\item[(iv)]
If $g$ is a non-zero strongly operator convex function on $(-\infty, 0)$, then $g^{(n)}(t)>0$ 
for $n=0,1,2,\cdots$. 
\end{itemize}
\end{Proposition}
\begin{proof}
(i). By Proposition\;\ref{Th:2}(ii) it is sufficient to verify this for $g(t)=\frac{1}{t-x}$, $x\le 0$, 
which is obvious.  
(ii). Since $1/f(t)$ is positive and operator decreasing, it is strongly operator convex; 
hence by (i) it is completely monotone.  
One can  show (iii) and (iv) in the same way as 
the proof of (i) by using (1) and Proposition\;\ref{Th:2}. 
\end{proof}

\section{Main results}

 \begin{Theorem}\label{Th:3}
 Let $f(t)$ be a continuous function on $J$ and $t_0\in J$. Then 
$f(t)$ is operator monotone if and only if $K_f(t, t_0)$ is strongly operator convex. 
\end{Theorem}
\begin{proof}
We may assume $f$ is $C^1$ on the interior of $J$, since both conditions imply this.  
 We first consider the case where $J=(a, \infty)$.  Assume $f(t)\in {\bf P}(J)$. Then $f(t)$ is operator 
concave, i.e., $-f$ is operator convex. This implies that $K_f(t, t_0)$ is 
operator decreasing on $(a, \infty)$ for each $t_0\in J$. By Proposition\;\ref{Th:2}(ii)  
$K_f(t, t_0)\in${\bf SOC}$(a, \infty)$. Assume conversely $K_f(t, t_0)\in${\bf SOC}$(a, \infty)$. 
If $K_f(t, t_0)\equiv 0$, then $f(t)$ itself is constant and hence operator monotone. If 
$K_f(t, t_0)\not\equiv 0$, then $K_f(t, t_0)>0$ for every $t$ and operator decreasing.  
Since $K_{-f}(t, t_0)$ is operator monotone, $-f$ is operator convex; and hence $f$ is operator concave.  So $f$ turns out to be operator monotone after all. 
The case where $J=(-\infty, b)$ can be shown 
in the analogous way, 
so we next consider the case where $J$ is a finite interval $(a, b)$.
  Assume $f\in{\bf P}(a, b)$ and decompose it as $f(t)=f_+(t) + f_-(t)$, as 
mentioned above Proposition\;\ref{Th:2}. 
Then $K_{f_+}(t, t_0)\in${\bf SOC}$(a, \infty)$,  $K_{f_-}(t, t_0)\in${\bf SOC}$(-\infty, b)$ 
for $t_0\in (a, b)$. We therefore get 
$K_{f}(t, t_0)= K_{f_+}(t, t_0) + K_{f_-}(t, t_0)\in${\bf SOC}$(a, b)$. 
 Assume conversely $K_{f}(t, t_0)\in${\bf SOC}$(a, b)$. 
We may assume $K_{f}(t, t_0)>0$ and hence $f'(t_0)>0$. Note that $f(t)$ has a 
holomorphic extension $f(z)$ to $J\cup \Pi_+ \cup \Pi_-$ since  $K_{f}(t, t_0)$ 
is operator convex. 
Since $h(t)\!:=1/K_{f}(t, t_0)$ is operator concave on $(a, b)$, $K_{h}(t, t_0)$ is 
operator decreasing. From the formula 
$$\frac{1}{f(t)- f(t_0)}=  K_{h}(t, t_0) + \frac{1}{f'(t_0)}\frac{1}{t-t_0}\quad (t\ne t_0),$$
by taking account of $f' (t_0)>0$, it follows that $\frac{1}{f(t)- f(t_0)}$ is operator 
decreasing on $(a, t_0)$ and $(t_0, b)$; hence $f(t)$ is operator monotone 
on $(a, t_0)$ and $(t_0, b)$. Thus the holomorphic function $f(z)$ on
 $J\cup \Pi_+ \cup \Pi_-$ must be a Pick function. Therefore $f(t)\in {\bf P}(J)$. 
Since a continuous function on $J$ is in ${\bf P}(J)$ or ${\bf SOC}(J)$ if and only if its restriction to the interior is in the correct class, it remains only to consider the case where $t_0$ is an endpoint of $J$.  
 We show just the following, because the other cases could be similarly 
shown. \\
\it Let $f(t)$ be continuous on $[a, b)$ and differentiable on $(a, b)$. Then 
$f(t)\in {\bf P}[a, b)$ if and only if $K_f (t, a)\in ${\bf SOC}$(a, b)$. \\
\rm Assume $f(t)\in {\bf P}[a, b)$. Then for every $c\in (a, b)$,  
$K_f (t, c)\!\in ${\bf SOC}$(a, b)$. Since $K_f (t, c)$ converges to $K_f (t, a)$ as 
$c\to a+0$ uniformly on every compact interval in $(a, b)$,  
$K_f (t, a)\in ${\bf SOC}$(a, b)$. 
Assume conversely $K_f (t, a)\in ${\bf SOC}$(a, b)$. If $K_f (t, a)\equiv 0$, then 
$f(t)$ is constant on $(a, b)$ and hence on $[a, b)$. So we need to consider 
only the case where $K_{f}(t, a)>0$ for every $t$. Since $h(t):=1/K_{f}(t, a)$ is 
operator concave on $(a, b)$, $K_{h}(t, t_1)$ is 
operator decreasing for every $t_1\in(a, b)$. 
Since 
$$K_{h}(t, t_1)= \frac{1}{f(t)- f(a)}\frac{t-a}{t-t_1} - \frac{1}{K_f (t_1, a)}\frac{1}{t-t_1}
\quad (t\ne t_1),$$
by taking account of $K_f (t_1, a)>0$ we see that $\frac{1}{f(t)- f(a)}\frac{t-a}{t-t_1}$
is operator decreasing on $(t_1, b)$.
Since $\frac{1}{f(t)- f(a)}\frac{t-a}{t-t_1} \to \frac{1}{f(t)- f(a)}$ as $t_1 \to a$,
 $\frac{1}{f(t)- f(a)}$ is operator decreasing on $(a + \epsilon, b)$ for arbitrary  
$\epsilon>0$, and hence on $(a, b)$. This consequently yields $f(t)\in {\bf P}[a, b)$. 
\end{proof}

\begin{Exm}\rm
It is well-known that $\tan t$ is operator monotone on $(-\frac{\pi}{2}, \frac{\pi}{2})$, 
but it has not been known even whether $\frac{\tan t}{t}$ is operator convex so far. 
However, by the above theorem, $\frac{\tan t}{t}$
 is strongly operator convex.
\end{Exm}

By using part of L\"{o}wner's Theorem, we can derive a differential criterion for strong operator convexity which is simpler than the one in \cite{B-2}.

\begin{Corollary}\label{Cor:3-1}
If $g$ is a continuous function on $J$ which is $C^1$ on the interior $J^{\circ}$ and if $t_0\in J$, then $g$ is strongly operator convex if and only if the kernel function $C_g$ is positive semi-definite on $J^{\circ}$, where 
$$C_g(t,s)=\frac{(t-t_0)g(t)-(s-t_0)g(s)}{t-s}\quad(t\ne s), \quad C_g(t,t)=g(t)+(t-t_0)g'(t).$$
\end{Corollary}
\begin{proof} Apply the theorem to $(t-t_0)g(t)$.
\end{proof}

\begin{Corollary}\label{Cor:3-2}
Let $J$ be an open interval. If $g_n \in {\bf SOC}(J)$ and $g_n(t)$ converges pointwise to $g(t)$ as 
$n\to \infty$, then $g\in {\bf SOC}(J)$.    
\end{Corollary}
\begin{proof} Take $c\in J$. Since 
$f_n(t)\!:= g_n(t) (t-c)$ is operator monotone and 
converges pointwise to  $f(t)\!:= g(t) (t- c)$, 
$f(t)$ is operator monotone; hence $g\in {\bf SOC}(J)$.
\end{proof}

\begin{Corollary}\label{Cor:3-3}\rm (\cite{B-2})\it \; 
Let $g(t)\in {\bf SOC}(a, b)$. Then there is a decomposition of $g(t)$ such that 
$g(t)= g_+(t) + g_-(t)$ for $t\in (a, b)$, where $g_+(t)\in {\bf SOC}(a, \infty)$ and $g_-(t) \in {\bf SOC}(-\infty, b)$.
In this case $g(t)$ is represented as 
$$g(t)= \alpha + \int_{(-\infty, a]}\frac{1}{t-x}d\nu_- (x) + \int_{[b, \infty)}\frac{1}{x-t}d\nu_+ (x), $$
where $\alpha \geq 0$, $\int_{(-\infty, a]}\frac{1}{|x| + 1}d\nu_- (x)<\infty$, 
$\int_{[b, \infty)}\frac{1}{|x| + 1}d\nu_+ (x)<\infty$. 
\end{Corollary}
\begin{proof}
Take $c\in (a, b)$ and put $f(t)= g(t)(t-c)\in 
{\bf P}(a, b)$.  Decompose $f(t)$ as $f(t)= f_+(t) + f_-(t)$, where 
$f_+\in{\bf P}(a, \infty)$, $f_-(t) \in {\bf P}(-\infty, b)$ and put
$g_+(t)\!:=\frac{f_+(t) -f_+(c)}{t-c} $ for $t\in (a, \infty)$ and 
$g_-(t)\!:=\frac{f_-(t) -f_-(c)}{t-c} $ for $t\in (-\infty, b)$.
Then $g_+(t)$ and $g_-(t)$ obviously satisfy the required properties. 
In view of Proposition\;\ref{Th:2}, one can get the integral representation.  
\end{proof}  
The next proposition is probably known to some people, but we haven't seen a proof, and we are including it for expository purposes.

\begin{Proposition}\label{Th:10}
Let $f(t)$ be a function on a finite interval $(a, b)$. Then
\begin{itemize}
\item[(i)]
If $f$ is operator concave and operator monotone on $(a, b)$, then 
$f$ has an extension $\tilde{f}$ to $(a, \infty)$ such that $\tilde{f}$ is 
operator concave and operator monotone on $(a, \infty)$. Further, 
if $\tilde{f}(\infty) < \infty$, then $-\tilde{f} + \tilde{f}(\infty)\in{\bf SOC}(a, \infty)$
 and hence $-f + \tilde{f}(\infty)\in{\bf SOC}(a, b)$.
\item[(ii)] 
If $f$ is operator convex and operator decreasing on $(a, b)$, then 
$f$ has an extension $\tilde{f}$ to $(a, \infty)$ such that $\tilde{f}$ is 
operator convex and operator decreasing on $(a, \infty)$. Further, 
if $\tilde{f}(\infty) > -\infty$, then $\tilde{f} -\tilde{f}(\infty)\in{\bf SOC}(a, \infty)$
 and hence $f - \tilde{f}(\infty)\in{\bf SOC}(a, b)$.
\item[(iii)]
If $f$ is operator convex and operator monotone on $(a, b)$, then 
$f$ has an extension $\tilde{f}$ to $(-\infty, b)$ such that $\tilde{f}$ is 
operator convex and operator monotone on $(-\infty, b)$. Further, 
if $\tilde{f}(-\infty) >-\infty$, then $\tilde{f} - \tilde{f}(-\infty)\in{\bf SOC}(-\infty, b)$
 and hence $f - \tilde{f}(-\infty)\in{\bf SOC}(a, b)$.
\item[(iv)] 
If $f$ is operator concave and operator decreasing on $(a, b)$, then 
$f$ has an extension $\tilde{f}$ to $(-\infty, b)$ such that $\tilde{f}$ is 
 operator concave and operator decreasing on $(-\infty, b)$. Further, 
if $\tilde{f}(-\infty) < \infty$, then $-\tilde{f} + \tilde{f}(-\infty)\in{\bf SOC}(-\infty, b)$
 and hence $-f + \tilde{f}(-\infty)\in{\bf SOC}(a, b)$.
\end{itemize}
\end{Proposition}
\begin{proof}
To prove (i) we use the integral representation of the operator monotone function $f$ given in (1) above.  But since $f$ is also operator concave, there is another integral representation:
\begin{equation}\label{eq:7}
f(t)=p+qt+rt^2+\int_{-\infty}^{a}(\frac{1}{x-t}-\frac{x+t-2t_0}{(x-t_0)^2})d\nu_-(x)
+\int_{b}^{\infty}(\frac{1}{t-x}+\frac{x+t-2t_0}{(t_0-x)^2})d\nu_+(x),
\end{equation}
for suitable choices of the constants $p$, $q$, and $r$ and the positive measures $\nu_{\pm}$.  There is a uniqueness result for such representations based on the theory of the Poisson kernel for the upper halfplane.  Using this and comparing the two integral representations, we see that $\nu_+ =0$ and that $\nu$ vanishes on $[b,\infty)$ (and in fact $\nu =\nu_-$).  The rest of (i) is clear and the proofs of the other parts are similar.
\end{proof}

\begin{Proposition}\label{Th:4}
Let $\varphi$ be an operator monotone function on an interval $J$ and $f$ a strongly operator convex function whose range lies in $J$.

(i) If $0\in J$ and $\varphi(0) \ge 0$, then $\varphi\circ f$ is strongly operator convex.

(ii) If either $0\in J$ or $0$ is the left endpoint of $J$. then $\varphi\circ f$ is operator convex.
\end{Proposition}
\begin{proof}
(i)  We use (1) for $\varphi$ with $t_0=0$.  Let $\varphi_x(t)=1/(x-t)-1/x=t/x(x-t)$ for $x\notin J$.  Since $\alpha=\varphi(0)\ge0$, it is sufficient to show $\varphi_x\circ f$ is strongly operator convex for each $x$.  By Lemma 2.2 this is equivalent to operator convexity of $-1/\varphi_x\circ f$.  Since $-1/\varphi_x\circ f = x-x^2/f$, this is clear.

(ii)  Again we use (1), but now  $t_0$ can be any point in $J$.  Since operator convexity is preserved by additive constants, it is enough to show $g_x\circ f$ is operator convex for each $x\notin J$, where $g_x(t) =1/(x-t)$.  If $x$ is to the right of $J$, this follows from the fact that $g_x$ is operator monotone and operator convex.  If $x$ is to the left of $J$, then $x\le0$, and the result follows from Lemma 2.2, since $f-x$ is strongly operator convex.
\end{proof}

\begin{Proposition}\label{Th:5}
Let $\varphi$ be a function on $J$. Then $\varphi$ is both operator convex and operator monotone 
if and only if the composite $\varphi\circ f$ is operator convex whenever $f$ is operator convex on an interval and 
the range of $f$ is contained in $J$. 
\end{Proposition}
\begin{proof}
The ``only if'' part is clear. To show the ``if'' part we may assume  $J=[-1, 1]$, because a function is operator monotone 
(or operator convex) on $J$ if it is operator monotone (or operator convex) on every finite closed subinterval of $J$.  
Suppose $\varphi\circ f$ is operator convex for every operator convex $f$ with the range 
in $J$. Since $f$ can be the identity function on $J$, $\varphi$ is an operator convex function, and hence 
it is continuous.  Put $f(t)= - t/(t-2)$ for $-\infty<t<1$. Since $f$ is operator convex on 
$(-\infty, 1)$,  $\varphi\circ f$ is operator convex on $(-\infty, 1)$. 
Since $\varphi\circ f(-\infty)<\infty$, it is operator monotone. 
Since $f^{-1}$ is operator monotone on $(-1, 1)$, $\varphi=(\varphi\circ f) \circ f^{-1}$ is operator monotone on 
$(-1, 1)$ as well. Since $\varphi$ is continuous, $\varphi\in {\bf P}[-1, 1]$.  
\end{proof}
\begin{Remark}\rm
It seems interesting that in Proposition 3.6 $\varphi$ need only be operator monotone whereas in Proposition 3.7 it must be both operator monotone and opeator convex.  But note that by Proposition 3.5(iii) the conditions on $\varphi$ in Proposition 3.7 could be restated without explicitly mentioning operator convexity.  Using Propostion 3.7 and Lemma 2.2 it is easy to see that any function $\varphi$ which satisfies the conclusion of either part of Propostion 3.6 must also satisfy the hypothesis.  It is also easy to find the functions $\varphi$ such that $f$ operator convex implies $\varphi\circ f$ strongly operator convex.  These are the operator monotone functions whose natural domain is unbounded to the left and which are non-negative on their natural domain.  An equivalent condition is that $\varphi$ must be both operator monotone and strongly operator convex.
\end{Remark}
\section{Some methods to construct new functions from old ones}

In \cite{U-2010} it has been shown that $\int f(t)dt $ is operator convex if 
$f(t)$ is operator monotone.  

\begin{Proposition}\label{Th:7} Let $g(t)$ be a strongly operator convex function on $J$.  
Then $\int g(t) dt $ is an operator monotone function. But the converse implication 
does not hold. 
\end{Proposition}
\begin{proof} By Corollary\;\ref{Cor:3-3} it is sufficient to show that $\int\frac{1}{x-t}dt$
 is operator monotone for  $x$ to the right of $J$ and 
$\int\frac{1}{t-x}dt$ is operator monotone for $x$ to the left of $J$.  
These facts are easily verified.   
 We give a counterexample for the converse implication. 
$\int \frac{1}{t^2} dt = - \frac{1}{t}$ is operator monotone on $(0, \infty)$, but 
$ \frac{1}{t^2}$ is not strongly operator convex, because it is not operator decreasing 
there.
\end{proof}

\begin{Proposition}\label{Th:9}
Let $h\ne 0$ be in ${\bf SOC}(0, \infty)$. Then 
\begin{itemize}
\item[(i)] $\frac{1}{th(t)}\in {\bf SOC}(0, \infty)$.
\item[(ii)] $th(t)\in {\bf P}(0, \infty)$.
\item[(iii)] $t/h(t)$ is operator convex on $[0, \infty)$.
\item[(iv)]
Let $\phi, \psi$ be  positive operator monotone functions on $(0, \infty)$. 
Then  $\phi(h)$, $\psi(\frac{1}{th(t)})$ and $\phi (h) \psi (\frac{1}{th(t)})$ 
are all in ${\bf SOC}(0, \infty)$.
\end{itemize}
\end{Proposition}
\begin{proof}
Since $1/h(t)$ is positive and operator monotone, $0\leqq \frac{1}{h(+0)}<\infty$ and 
$\frac{1}{t} (\frac{1}{h(t)} - \frac{1}{h(+0)})\in {\bf SOC}(0, \infty) $ 
by Theorem \;\ref{Th:1} and Corollary\;\ref{Cor:3-3}. 
Because $\frac{1}{t}\frac{1}{h(+0)}\in {\bf SOC}(0, \infty)$, we obtain 
$\frac{1}{t} \frac{1}{h(t)}\in {\bf SOC}(0, \infty) $, namely (i); and hence we get (ii) as well. 
Since $1/h(t)\in {\bf P}[0, \infty)$ and $ 1/h(t)=K_{t/h}(t, 0)$, $\frac{t}{h(t)}$ 
is operator convex on $[0, \infty)$. This is (iii). 
We finally show (iv). Since $h$ and $\frac{1}{th(t)}$ are operator decreasing, so are 
 $\phi(h)$ and $\psi(\frac{1}{th(t)})$. We have only to show 
$\phi (h) \psi (\frac{1}{th(t)})$ is operator decreasing. This is equivalent to 
the operator monotonicity of  $\phi (h(1/t)) \psi (\frac{1}{h(1/t) 1/t})$, 
which follows from  Lemma 2.1 of \cite{U-F} (cf. Prop. 7.16 of \cite{S-S-V}) 
since  $h(1/t)$ and $\frac{1}{h(1/t) 1/t}$ are both positive operator monotone 
functions on $(0, \infty)$. 
\end{proof}

We now proceed with a main construction for making new operator monotone functions from old ones. Start with a non-constant operator monotone function $f_0$ on $J$. Choose $t_0$ in $J$ and let $f_1(t)=(f_0(t)-f_0(t_0))/(t-t_0)$. Since $f_1$ is a non-zero strongly operator convex function, we can define a (strictly negative) operator convex function $f_2$ by setting $f_2=-1/f_1$. Then choose $t_1$ in the domain of $f_2$ and let $f_3(t)=(f_2(t)-f_2(t_1))/(t-t_1)$, a new operator monotone function.
\begin{Exm}\rm
Let $f_0(t) = \tan t$ as in Example 3.1 and take $t_0=t_1=0$.  Then $f_3(t)=(\tan t-t)/t\tan t$.  Although one could presumably verify directly that $f_3$ is a Pick function, the fact that our construction can easily produce $f_3$ and guarantee that it is operator monotone may be interesting.
For another example take $f_0(t)=t^\alpha$, $t\ge 0$, where $0<\alpha<1$, $t_0=1$ and $t_1=0$. Then $f_3(t)=(t^{\alpha-1}-1)/(t^\alpha-1)$ on $(0,\infty)$. The direct proof that $f_3$ is operator monotone, by checking that it is a Pick function, involves a somewhat non-trivial calculus problem. 
\end{Exm}

The process can be continued to obtain an infinite sequence of functions such that each $f_{3n}$ is operator monotone, each $f_{3n+1}$ is strongly operator convex, and each $f_{3n+2}$ is operator convex, provided that no $f_{3n+1}$ is 0. It is not hard to see that this will be the case unless $f_0$ is rational. If $f_0$ is rational, then the degree of $f_{3n+3}$, defined as the maximum of the degrees of the numerator and denominator, will be one less than the degree of $f_{3n}$, and the process will eventually terminate. The construction depends on an infinite sequence $\{t_i\}$ of points used in the transitions from $f_{3n}$ to $f_{3n+1}$ and from $f_{3n+2}$ to $f_{3n+3}$. The requirements for this sequence are a little complicated.  Since a negative convex function always has a finite limit at any finite endpoint of its domain, the functions $f_{3n+2}$ may always be considered to be defined on $\bar{J}$.  Thus $t_i$ for $i$ odd, which is used in the transition from $f_{3n+2}$ to $f_{3n+3}$, is always allowed to be an endpoint of $J$.  But if such a $t_i$ is an endpoint, then in general $t_{i+1}$ cannot be the same endpoint.

Another construction 
starts with an operator monotone function $f^*_0$. Then the strongly operator convex function $f^*_1$ is defined, as before, by $f^*_1(t)=(f^*_0(t)-f^*_0(t_0))/(t-t_0)$. But now we define a new operator monotone function $f^*_2$ by $f^*_2(t)=(f^*_1(t)-f^*_1(t_1))/(t-t_1)$. We can continue the process to obtain an infinite sequence which alternates between operator monotone functions and strongly operator convex functions. However, the theory of strongly operator convex functions is not really being used, since we need only the fact that $f^*_{2n+1}$ is operator convex.  Also recall that since the measures, as in (1) and Corollary 3.4, for $f^*_{n+1}$ are obtained from those for $f^*_n$ by multiplying by $1/|x-t_n|$, all of the measures will be finite for $n\ge 2$, and they will have increasingly many finite moments as $n$ increases. Thus the operator monotone functions $f^*_{2n}$ will not be ``general'' but will have special properties. Finally, there is one more reason why we think that the $\{f^*_n\}$ process is less interesting than the original one. Suppose $f$ is a strictly positive continuous function on $J$ and $g=-1/f$. Then $g$ operator convex implies $f$ operator convex but the converse is false. Therefore if $h_1(t)=(g(t)-g(t_1))/(t-t_1)$ and $h_2(t)=(f(t)-f(t_1))/(t-t_1)$, then $h_1$ operator monotone implies $h_2$ operator monotone but not conversely. Therefore it is presumably more interesting to know that $h_1$ is operator monotone than to know that $h_2$ is operator monotone. If $f$ is the function $f_1$ or $f^*_1$ in the above constructions, then the first construction produces $h_1$ and the second produces $h_2$.
\begin{Exm}\rm
If $f^*_0(t) =\tan t$ and $t_0=t_1=0$, then $f^*_2(t)=(\tan t-t)/t^2$, and if 
$f^*_0(t)=t^\alpha$, $t_0=1$, and $t_1=0$. Then $f^*_2(t)=(t^{\alpha-1}-1)/(t-1)$ on $(0,\infty)$.  Although we believe that this process is less interesting than the previous one, this example suggests that it may still be interesting.
\end{Exm}

It is also possible to run the first process backwards, though this is slightly problematical. Thus start with an operator monotone function $f_0$ on $J$ and let $f_{-1}(t)=f_0(t)(t-t_0)+c_0$. It is necessary to choose $c_0$ small enough that $f_{-1}(t)<0$, $\forall t$, and this is not always possible. But it is possible if both $J$ and $f_0$ are bounded. So if necessary we can restrict $f_0$ to a smaller interval in order to construct a suitable operator convex function $f_{-1}$. Then let $f_{-2}=-1/f_{-1}$, and take $f_{-3}(t)=f_{-2}(t)(t-t_1)+c_1$. So $f_{-3}$ is a new operator monotone function.  It is permissible for some of the $t_i$'s to be endpoints of $J$ even if those endpoints are not in $J$. But if we want really to be running the original process backwards, it is necessary that the singletons $\{t_i\}$ have measure 0.

It might be interesting to investigate the behavior of $\{f_n\}, \{f^*_n\}$, and $\{f_{-n}\}$ as $n\to\infty$. We have not attempted this, since we have no expertise in dynamical systems, but the situation for $\{f^*_n\}$ is probably not very difficult. For $n\ge 2$ the $f^*_n$'s are determined completely by the measures appearing in (1) or Corollary 3.4 (the $1/(x-t_0)$ term in (1) can be omitted when $n\ge 2$), and it is easy to calculate the measures in terms of the measure for $f^*_0$ and the $t_i$'s. If the length of $J$ is 2 and we choose $t_i$ to be the midpoint of $J$ for all $i$, then $\{f^*_{2n}\}$ and $\{f^*_{2n-1}\}$ will have possibly non-zero limits as $n\to\infty$.

There are other ways to derive new operator functions from old ones.  One possibility is to start with an operator monotone function $f$, define an operator convex function $g$ by $g(t) = (t-t_0)f(t)$, and then a new operator monotone function $h$ with $h(t) =(g(t)-g(t_1))/(t-t_1)$.  This technique was used by Hansen and Pedersen to derive \cite[Theorem 3.9]{H-P}.  One could also use the same technique to derive a new strongly operator convex function $h$ from a given strongly operator convex function $f$.  This step could then be inserted into one of the processes given above.  Or, in the main process, instead of defining $f_{3n+2}=-1/f_{3n+1}$, one  could take $f_{3n+2}=\varphi\circ f_{3n+1}$, where $\varphi$ satisfies part (ii) of Proposition 3.6.  Clearly there are so many ways to move from one of the three classes of functions to another (or to the same class) that it would be pointless to try to list them all.  
 

\end{document}